\documentclass[twoside,11pt]{article}

\usepackage[preprint,nohyperref]{jmlr2e}
\usepackage{lipsum}
\usepackage{graphicx}
\usepackage{epstopdf}
\usepackage{hyperref}
\usepackage{algorithmic}
\usepackage[framed,numbered,autolinebreaks,useliterate]{mcode}
\usepackage{amsmath}
\usepackage{amssymb}
\usepackage{amsfonts}
\usepackage{algorithm}
\usepackage{bbm}
\usepackage{mathtools}
\usepackage{cleveref}
\usepackage{listings}
\usepackage{color}
\usepackage{tikz}
\usepackage{enumitem}
\usepackage{amsopn}
\usetikzlibrary{shapes,arrows}
\newcommand*\diff{\mathop{}\!\mathrm{d}}

\numberwithin{equation}{section}
\numberwithin{table}{section}
\numberwithin{figure}{section}
\numberwithin{theorem}{section}
\numberwithin{lemma}{section}

\crefname{lemma}{lemma}{lemmas}

\newcommand{\reel}[1]{\mathbb{R}}

\ShortHeadings{Multilevel Monte Carlo learning}{Gerstner, Harrach, Roth, and Simon}
\firstpageno{1}

\begin{document}

\title{Multilevel Monte Carlo learning}

\author{\name Thomas Gerstner \email gerstner@math.uni-frankfurt.de \\
		\name Bastian Harrach \email harrach@math.uni-frankfurt.de \\
		\name Daniel Roth \email roth@math.uni-frankfurt.de \\
       \addr Department of Mathematics\\
       Goethe University Frankfurt\\
       60325 Frankfurt am Main, Germany
       \AND
       \name Martin Simon \email  \\
       \addr Deka Investment GmbH, Germany
       }

\editor{}

\let\thefootnote\relax\footnotetext{The opinions expressed in this article are those of the authors and do not necessarily reflect views of Deka Investment GmbH.}\let\thefootnote\svthefootnot

\maketitle

\begin{abstract}
%
%
%
In this work, we study the approximation of expected values of functional quantities on the solution of a stochastic differential equation (SDE), where we replace the Monte Carlo estimation with the evaluation of a deep neural network. Once the neural network training is done, the evaluation of the resulting approximating function is computationally highly efficient so that using deep neural networks to replace costly Monte Carlo integration is appealing, e.g., for near real-time computations in quantitative finance. However, the drawback of these nowadays widespread ideas lies in the fact that training a suitable neural network is likely to be prohibitive in terms of computational cost. We address this drawback here by introducing a multilevel approach to the training of deep neural networks. More precisely, we combine the deep learning algorithm introduced by Beck et al. \cite{beck2018solving} with the idea of multilevel Monte Carlo path simulation of Giles \cite{giles2008multilevel}. The idea is to train several neural networks, each having a certain approximation quality and computational complexity, with training data computed from so-called \textit{level estimators}, introduced by Giles \cite{giles2008multilevel}. We show that under certain assumptions, the variance in the training process can be reduced by shifting most of the computational workload to training neural nets at coarse levels where producing the training data sets is comparably cheap, whereas training the neural nets corresponding to the fine levels requires only a limited number of training data sets. We formulate a complexity theorem showing that the multilevel idea can indeed reduce computational complexity.
\end{abstract}

\begin{keywords}
  deep learning, multilevel, Monte Carlo, computational complexity
\end{keywords}

\section{Introduction} 
Consider a multi-dimensional SDE 
\begin{align}
dS(t)=\mu(S,t)\diff t + \sigma(S,t) \diff W(t), \text{ \hspace{3mm} } 0<t\le T,\label{sde-generic}
\end{align}
with initial value $S_0$, drift $\mu(S,t)$ and volatility $\sigma(S,t)$, which under certain conditions has a pathwise unique strong solution $S(t)$, see, e.g., \cite{kloeden2012numerical}. We are interested in the expected value $V(S(T))$, where $V(S)$ is a scalar function (payoff) and $S$ the solution of the above SDE. The Milstein discretization of the SDE with step-width $h$ is of the form
\begin{align}
\begin{split}\hat{S}_{n+1}=\hat{S}_n&+\mu(\hat{S}_n,t_n)h+ \sigma(\hat{S}_n,t_n) \sqrt{h}\Delta Z_n \\&+ \frac{1}{2}\sigma(\hat{S}_n,t_n) \sigma'(\hat{S}_n,t_n)\left( (\sqrt{h} \Delta Z_n)^2-h\right),\end{split} \label{milstein}
\end{align}
with $\Delta Z_n$ i.i.d. standard normal for $n=0,\dots, T/h-1$, $t_n=nh$, and $\hat{S}_0=S_0$.
We denote the approximation of $V(S(T))$ using step-width $h$ as follows:
\begin{align}
P_h:=V(\hat{S}_{T/h}).\label{singlemc}
\end{align}
The mean of the sampled payoff values given by $N^{-1}\sum_{i=1}^N P_h^{(i)}$ from $N$ independent path simulations, is the simplest estimate of $\mathbb{E}[V(S(T))]$. Assuming certain conditions on the drift, volatility and payoff, see e.g. \cite{kloeden2012numerical}, the estimators' mean squared error (MSE) is asymptotically of the form
\begin{align*}
\text{MSE} \approx c_1N^{-1}+c_2h^2,
\end{align*}
with positive constants $c_1,c_2$.
Hence, we can achieve an error bound of $\mathcal{O}(\epsilon^{2})$, for any $\epsilon>0$, at a computational complexity of order $\mathcal{O}(\epsilon^{-3})$ for the MSE, i.e., it would require that $N=\mathcal{O}(\epsilon^{-2})$ and $h=\mathcal{O}(\epsilon^{-1})$.
\cite{giles2008multilevel} introduced the idea of multilevel Monte Carlo simulation, which achieves a complexity reduction. Under certain conditions and assuming, e.g., a Lipschitz-continuous payoff and the Milstein scheme, the complexity can be reduced to $\mathcal{O}(\epsilon^{-2})$, see \cite{giles2008improved}.

Let ${P}_l$ denote the approximations, defined by \eqref{singlemc}, using the discretizations $h_l=2^{-l}T$ for $l=0,1,\dots,L$. 
We have
\begin{align*}
\mathbb{E}[{P}_{h_L}]= \mathbb{E}[{P}_{h_0}]+ \sum_{l=1}^L \mathbb{E}[{P}_{h_l}-{P}_{h_{l-1}}]
\end{align*}
and the multilevel Monte Carlo idea is to independently estimate each of the expectations on the right-hand side.
Therefore, consider the following so-called \textit{level estimators} for $\mathbb{E}[P_{h_0}]$ and $\mathbb{E}[{P}_{h_l}-{P}_{h_{l-1}}]$ defined by
\begin{align}
\hat{Y}_l= 
\begin{cases}
N_0^{-1} \sum_{i=1}^{N_0} {P}_{h_0}^{(i)} &\text{ for } l=0, \\
N_l^{-1}\sum_{i=1}^{N_l} \left( {P}_{h_l}^{(i)} -{P}_{h_{l-1}}^{(i)} \right),& \text{ for } l>0, \\
\end{cases} \label{levelmc}
\end{align}
using $N_l$ paths and where the two discrete approximations ${P}_{h_l}^{(i)} $ and ${P}_{h_{l-1}}^{(i)} $ come from the same Brownian path, such that the difference ${P}_{h_l}^{(i)} -{P}_{h_{l-1}}^{(i)}$ is often small due to the strong convergence properties of the Milstein scheme, see, e.g., \cite{giles2008multilevel} for further explanations. The final  multilevel estimator $\hat{Y}$ is given by the sum of the \textit{level estimators}:
\begin{align}
\hat{Y}=\sum\limits_{l=0}^L \hat{Y}_l.\label{multimc}
\end{align}

The complexity reduction of the multilevel approach depends on the \textit{level estimators} \eqref{levelmc}. Giles presented a quite general complexity theorem that can be applied to a variety of financial models and payoffs without using a specific numerical approximation scheme. Further studies are e.g. made in \cite{giles2008improved}.

In practice, the volatility term $\sigma(\cdot,\cdot)$ is persistently re-calibrated as it depends on the market-implied volatility. Instead of contributing new price computations arising from an updated volatility term, this work's key motivation is to replace these by estimating an appropriately trained neural network $\mathcal{N}$. We refer to \cite{higham2019deep} for an in-depth introduction to neural network training. As a first focus in this work, we will study different approaches to generate the necessary training data and correctly choose the input parameters.

The neural network maps the parameter vector (inputs) to the expected value of a payoff function (output). Thus, we reformulate the issue as a suitable stochastic optimization problem and solve it through an artificial neural network approximation.
We extend the above considerations to a set of stochastic differential equations and to a family of payoff functions. Consider the multi-dimensional stochastic differential equation
\begin{align}
dS(t)=\mu(S,t,a)\diff t + \sigma(S,t,b) \diff W(t), \text{ \hspace{3mm} } 0<t\le T,\label{sde-parameter}
\end{align}
with initial value $S(0)=s_0\in \mathbb{R}_+$, time of maturity $T\in\mathbb{R}_+$, drift $\mu(S,t,a)$ and volatility $\sigma(S,t,b)$, with $a=(a_1,\dots,a_m)\in \mathbb{R}^m$ and $b=(b_{1},\dots,b_{s})\in \mathbb{R}^s$.

Let the stochastic process ${S}_{a,b,s_0,T}$ be the solution of the SDE \eqref{sde-parameter} defined by the parameters $a,b,s_0$ and $T$. Consider a family of payoff functions $V(S,v)$, with $v=(v_1,\dots,v_r) \in \mathbb{R}^r$. Then, we will be interested in the expected value of
\begin{align}
P:y \mapsto V\left({S}_{a,b,s_0,T}(T),v\right),\label{py}
\end{align}
for fixed
\begin{align}
y:=(a,b,v,s_0,T)\in Y  \subset \mathbb{R}^{V}\times \mathbb{R}_+^2,\label{training-intervall}
\end{align}
with $V=m+s+r$. We call $Y$ the training set, i.e. it contains all parameter vectors of interest.
Finally, we will be interested in
\begin{align}
\bar{P}:y \mapsto \mathbb{E}[P(y)],\label{epy}
\end{align}
the input (model and payoff parameters) to price (expected value of the payoff) map and we aim to find an appropriate neuronal network for its approximation, i.e. a network $\mathcal{N}:Y\rightarrow \mathbb{R}$ minimizing
\begin{align}
\left\lVert \bar{P}(y)-\mathcal{N}(y)\right\rVert_{L^p},\label{approach1}
\end{align}
for $1\le p \le \infty$.

One intuitive approach to efficiently generate training data (outputs) for learning such a network is to use, e.g., a multilevel Monte Carlo estimator for reasonably chosen or randomly selected input values (for a high-dimensional case). One chooses a proper set of inputs of the form \eqref{training-intervall} and estimates the outputs \eqref{epy} at the required accuracy.

However, \cite{beck2018solving} presented an alternative approach: rather than using estimated prices for chosen input values, they use the individual sampled paths of the single level estimator \eqref{singlemc} for randomly selected input values. In other words, one uses randomly sampled inputs and estimates single paths of \eqref{py}, instead of \eqref{epy}, as outputs.
Within this work, we will compare both approaches with respect to computational complexity.

This work's main idea is to use the advantages of the multilevel Monte Carlo path simulation and combine it with the procedure of using only single paths as training data. I.e., we will present an approach, which trains several neural networks $\mathcal{N}_l:Y\rightarrow \mathbb{R}$, for $l=0,1,\dots,L$, where each network uses paths of the multilevel \textit{level estimators} \eqref{multimc} as training data. We will call this approach multilevel Monte Carlo learning.

Therefore, lets shortly extend the ideas of the Milstein payoff discretization \eqref{singlemc} and the multilevel \textit{level estimators} \eqref{multimc} to the reformulated optimization probelem \eqref{approach1} on a training set $Y$.

Consider using the Milstein scheme \eqref{milstein}, with step-width $h$, as an approximation of the solution of SDE \eqref{sde-parameter}, leading to 
\begin{align}
{P}_h:y\mapsto V\left(\left(\hat{S}_{{a,b,s_0,T}}\right)_{T/h},v\right),\label{milsteinp}
\end{align}
as an approximation of \eqref{py}.

Similar ideas lead to paths of the \textit{level estimators} on the training set given by 
\begin{align}
\hat{Y}_l: y\mapsto 
\begin{cases}
 {P}_{h_{0}}(y)&\text{ for } l=0, \\
  {P}_{h_{l}}(y)-{P}_{h_{l-1}}(y)),& \text{ for } l>0. \\
\end{cases} \label{levelmc2}
\end{align}
Hence, instead of using \eqref{milsteinp} for the training data, the training data for each network $\mathcal{N}_l$ will be simulated by using the paths \eqref{levelmc2}. All things considered, the multilevel Monte Carlo training approach searches the approximation
\begin{align}
\hat{\mathcal{N}}:= \sum_{l=0}^L \mathcal{N}_l, \label{multinet-intro}
\end{align}
which minimize \eqref{approach1}.

The rest of this work's structure is as follows: First of all, we will present an introductory example to show a possible strength and an easy implementation for which no theoretical knowledge is needed.

Section 2 will briefly review and compare both of the approaches mentioned above, using one neural network with respect to the computational complexity. We will call these single-level approaches. The derived complexities will be the references for this work's primary goal, further complexity reductions. Then, the main complexity theorem and an extension of numerical examples will be presented. Sections 3 and 4 will give some possible extensions and a conclusion.
In the appendix, we will discuss some mathematical background and present the proof. Furthermore, we will give some numerical results, studying the convergence with respect to the batch-size and some variance reduction effects.

\subsection{Introductory example}\label{sec-example}

We start by presenting an easy and short procedure to implement the multilevel approach, to which no further theoretical background knowledge is necessary for its application. 

Consider the following procedure to train a neural network to achieve the required accuracy. The practitioner chooses a promising network structure, fixes a limit for the number of training steps, chooses a batch-size (training data used in each training step), and uses the training algorithm as stated in \cite{beck2018solving}. The approach computes the needed training data in each training step by using single paths, as in \eqref{milsteinp}, on randomly selected inputs. If the resulting network does not achieve the required accuracy, the training process is repeated from scratch using increased batch-size, increased training steps, or adjusted network structure. 

The examples in this section will use the following variation: Though, using a fixed amount of training steps and fixed network structure for all training processes, we gradually increase the batch-size until the required accuracy is achieved. We will use a five-dimensional training set from a geometric Brownian motion's initial parameters since a closed solution is available for this example. 

\subsubsection{Single-level learning example}\label{single-level-example}

First, we train a single neural network, by successively increasing the batch-size, until the required accuracy is achieved. Consider a geometric Brownian motion
\begin{align}
\diff S(t)&=\mu S(t) \diff t + \sigma S(t) \diff W(t), \text{ \hspace{3mm} } 0<t\le T,\label{bs2}
\end{align}
with initial value $S(0)=s_0$, constant drift $\mu$ and constant volatility $\sigma$.

Furthermore, consider a European call option given by
\begin{align}
V(S,K):=\max\{S-K,0\},\label{europ-option}
\end{align}
with fixed strike price $K$. Hence, we want to approximate
\begin{align*}
P:y\mapsto V(S_{\mu,\sigma,s_0,T}(T),K)
\end{align*}
with $y=(\mu,\sigma,s_0,T,K)\in Y$.
Technically, the final price is denoted as the discounted expected payoff, which, for simplicity, we omit here.

Consider the training set $Y\subset \mathbb{R}^5$ to be of the specific form: 
\begin{align*}
Y=[0.02,0.05]\times[0.1,0.2]\times [80,120]    \times [0.9,1.0] \times [109,110].
\end{align*}
It is well known that the GBM \eqref{bs2} together with payoff \eqref{europ-option} lead to a closed solution for the option price for each $y\in Y$, see e.g. \cite{hull}. 

The idea is to train a neural network to evaluate the expected value up to a required error $\epsilon$ for each input vector $y$ within the training set $Y$. 
E.g., requiring accuracy of $\epsilon=0.01$ ($L^\infty$-error) of the trained network we use the feasible step-width $h=1/128$ for the Milstein scheme and the paths discretiazion \eqref{milsteinp}. 

As introduced above, we will use the algorithm of Beck et al., to which the PYTHON source code can be found in section 4 of \cite{beck2018solving}. The code uses the open-source software library TensorFlow. For our tests, we slightly modify the code: instead of using fixed learning-rate boundaries, we use exponential decay, which, however, would deliver similar results for specific decay parameters.

We obtained feasible results using network structure and training parameters for the introduced training set, as stated in \cref{parameter_net}, see \cite{higham2019deep} for further parameter explanations. The results of the training processes for increasing batch-sizes can be found in \cref{results1}. The computations were made on an Nvidia K80 GPU and were repeated $10$ times for each batch-size. The $L^\infty$ error was estimated by comparing the approximation with the closed solution on $2.000.000$ randomly selected initial parameter vectors.

\begin{table}
\centering
  \begin{tabular}{lr}
  \hline
 Parameter& Value\\
 \hline
 neurons & $(50,50,1)$\\
 decay rate & $0.1$\\
 initial learning rate & $0.01$\\
  step rate & $40.000$\\
  training steps & $150.000$\\
 \hline
 \end{tabular}
 \caption{Structure and learning rate parameters for each neural network.}
  \label{parameter_net}
\end{table}

\begin{table}
 \small
\centering
  \begin{tabular}{c|c|c}
  \hline
 batch-size & mean of $L^\infty$  & time (hours)   \\
 \hline
 	125.000 & $0.0273$&  $2.32$h \\
   	 500.000  &	$0.0180 $ & $7$h 	\\
   	 2.000.000 & $0.0119 $ & $26.66$h \\
 \hline
 \end{tabular}
 \caption{Mean of the $L^\infty$-error for the single-level algorithms using the Milstein discretization \eqref{milstein}, whereat the simulations were repeated $10$ times. The last column shows the computation time using a Nvidia K80 GPU.  }
  \label{results1}
\end{table}

We see that a batch-size of $2.000.000$ achieves a feasible accuracy, to which the training process took $26.66$ hours.

\subsubsection{multilevel learning example}\label{multilevel-example}

Now, we explain the procedure for the multilevel approach. To present easy comparable results, we use the same network structure and amount of training steps used for the single-level approach for each of the networks we will train. Furthermore, we will use the same model setting and training set, as in the example above. However, we will use individualized batch-sizes. We will increase these likewise until the required accuracy is achieved.

We use the paths of the \textit{level estimators}, as described in \eqref{levelmc2}, to compute the training data in each training step. We present the modification of the code of Beck et al. used for the simulation of the \textit{level estimators'} paths, e.g., for $\hat{P}_1^{}-\hat{P}_{0}^{}$, in listing~\ref{algorithm2} in the appendix.

The used batch-sizes $M_l$ for each net and the amount of levels/nets $L$  are calculated with the multilevel sample estimator of Giles suggested in \cite{giles2008multilevel} (for which MATLAB codes can be found in \cite{gilescode}). The code originally estimates the Monte Carlo sample-sizes $N_l$ and amount of levels $L$ needed for the Multilevel approach to achieve a required accuracy $\epsilon$. Let us present an example of its application. For e.g. $\epsilon=0.01$ and $y=(0.05,0.2,100,1,110)$ the code delivers parameters as stated in \cref{batches-multilevel}.

\begin{table}
 \small
\centering
  \begin{tabular}{c|c|c|c|c|c|c|c|c}
  level $l$ & $0$ & $1$ & $2$ & $3$  & $4$ & $5$  & $6$  & $7$  \\  
\hline
$N_l$ & $3.000.000$ & $72695$ & $27756$ & $ 10550$ & $3691$ & $ 1308$ & $476 $ & $182 $\\
\hline 

 \end{tabular}
 \caption{Estimated needed Monte Carlo samples $N_l$ for the Multilevel Monte Carlo approach for $y=(0.05,0.2,100,1,110)$.}
  \label{batches-multilevel}
\end{table}

For the multilevel approach, we use these sample-sizes in the following way: We use the multilevel sample-ratios $N_l/N_0$, for $l=1,\dots L$, as the multilevel batch-size-ratios. Furthermore, we use the amount of multilevel levels $L$ as the amount of multilevel networks. 

I.e. we will multiply the above ratios with the initial batch-size $M_0$ used for network $\mathcal{N}_0$. Hence, we obtain $M_l=M_0\cdot N_l/N_0$ for $l=1,\dots L$ as the batch-sizes of the networks $\mathcal{N}_l$.

If we, e.g., study the ratios in \cref{batches-multilevel}, we see that the ratio between level one and level zero is given by $N_1/N_0=72.695/3.000.000\approx 0.024$. Now, by multiplying this ratio to an initial batch-size of e.g. $M_0=75.000$ delivers the batch-size $M_1=1817$, for the net $\mathcal{N}_1$ on level $l=1$. With increasing initial batch-sizes $M_0$, we obtain the batch-sizes for the remaining networks as presented in \cref{batches-intro}, at which, for a certain initial batch-size, we summarized the set of batch-sizes to a so-called multilevel id.

\begin{table}
 \small
\centering
  \begin{tabular}{c|c|c|c|c|c|c|c|c|c}
  \hline
 multilevel id &level $l$ & $0$ & $1$ & $2$ & $3$  & $4$ & $5$  & $6$  & $7$  \\  
\hline
1&$M_l$ & $75.000$ & $1817$ & $690$ & $264 $& $93 $& $ 33$& $ 12$& $ 5$  \\
2&$M_l$ & $300.000$ & $7268$ & $2760$ & $ 1056$& $ 372 $& $ 132 $& $ 48 $& $20 $  \\
3&$M_l$ & $1.200.000$ &$29072$ &$11040$&$4224$&$ 1488$&$528$&$192$& $ 80$ \\
\hline 
 \end{tabular}
 \caption{Estimated needed batch-size for the training of the specific level nets for $l=0,\dots,7$. }
  \label{batches-intro}
\end{table}

The multilevel training results using the estimated amount of levels and batch-sizes and its comparison to the single-level approach are given in \cref{results2}.

\begin{table}
 \small
\centering
  \begin{tabular}{c|l|c| l}
   \multicolumn{2}{c|}{single-level} & \multicolumn{2}{c}{multilevel}  \\
     \hline
 batch-size   & mean error (time)  &  id &  mean error (time)     \\
 \hline
  	125.000   & $0.0273 $ $(2.32\text{h}) $  &  1 &$ 0.0290$  $(4.15\text{h})$	\\
  	500.000   & $0.0182 $ $(7\text{h})$  & 2 &$ 0.0184$  $(5.29\text{h})$	\\
  	2.000.000  & $0.0119 $ $(26.66\text{h})$  & 3 & $ 0.0103$  $(11.18\text{h})$	\\
 \hline
 \end{tabular}
 \caption{Mean of the $L^\infty$-error for the single-level algorithm using the Milstein discretization (second column). The fourth column shows the mean of the $L^\infty$-error for the multilevel algorithm using respective batch-sizes as presented in \cref{batches-intro}. All simulations were repeated $10$ times. The brackets show needed computation time using a Nvidia K80 GPU.  }
  \label{results2}
\end{table}

Using the network structure, learning rates, and training steps of the single-level approach, we see that the multilevel approach leads to significant time-saving. 
For the lowest chosen batch-size ($125.000$), we see that single-level learning is slightly faster than multilevel learning. A reason for this could be that learning $8$ neural networks results in higher basic costs. However, if we require a higher error bound, e.g., as achieved for this low batch-size, the above-introduced procedure would suggest fewer neural networks.

\section{Neural network training}

In this section, we will study the error sources and computational costs. Therefore, we start with a more profound introduction to neural network training. Then, we compare the first introduced approach and the approach used in \cref{single-level-example} with respect to computational complexity. The second subsection will introduce the multilevel approach and present the computational complexity theorem.

\subsection{Single-level: error sources and complexity}

First, we want to discuss some neural networks' properties and their training processes, such as error sources and computational cost. For a more general introduction to neural networks, we refer to \cite{higham2019deep}. This work aims to train a neural network in such a way that it is capable of evaluating the expected value up to a required error $\epsilon$ for each input vector within the training set $Y$.

We study the training process of a generic artificial neural network (see e.g. \eqref{background} for the definition) which is given by the following series of functions
\begin{align}
\mathcal{N}_{\nu,\theta_i,{P}_h,{\mathbb{L}}_M}: Y\rightarrow \mathbb{R},\label{net1}
\end{align}
for $i=1,\dots,K$, with initial weights $\theta_0$. I.e. each of the $K$ training steps modifies the weights $\theta_i\in\mathbb{R}^{\nu}$ and hence defines a new element of the series. For each training process, we fix the network structure $\nu$ (layers and neurons), the approach ${P}_h$ to generate the training data and the amount of samples $M$ (batch-size) used to evaluate the loss function ${\mathbb{L}}_M$.

I.e., we will neither modify the estimation ${P}_h$ nor the batch-size $M$ during each of the training processes.

Demanding a specific error bound $\epsilon$, the challenge consists of choosing the most efficient parameters $\nu, M,h$, and $K$. 

Since, using a stochastic gradient descent algorithm, we will study the expectation of \eqref{approach1}. Therefore, we will be interested in finding certain weights after $K$ trainings steps satisfying
\begin{align}
\mathbb{E}\left[\left\lVert \bar{P}(y)-\mathcal{N}_{\nu,\theta_K,{P}_h,{\mathbb{L}}_M}(y)\right\rVert_{L^1}^2\right] <\epsilon^2,\label{single-mse}
\end{align}
for a required error $\epsilon>0$. Using the standard variance expansion, the left hand side can be expanded to
\begin{align}
\mathbb{V}\left[\left\lVert \bar{P}(y)-\mathcal{N}_{\nu,\theta_K,{P}_h,{\mathbb{L}}_M}(y)\right\rVert_{L^1}\right]+
\mathbb{E}\left[\left\lVert \bar{P}(y)-\mathcal{N}_{\nu,\theta_K,{P}_h,{\mathbb{L}}_M}(y)\right\rVert_{L^1}\right]^2.\label{single-mse2}
\end{align}
We specify the error sources by further decomposing the inner term of \eqref{single-mse} to obtain
\begin{align}
\left\lVert \bar{P}(y)-\mathcal{N}_{\nu,\theta_K,{P}_h,{\mathbb{L}}_M}(y)\right\rVert_{L^1}
\stackrel{\text{}}{\le}& \left\lVert \bar{P}(y)-\mathbb{E}[{P}_h(y)]\right\rVert_{L^1}\label{error-disk-sde}\tag{e1}\\
\stackrel{\text{}}{+}&  \left\lVert \mathbb{E}[{P}_h(y)]-\mathcal{N}_{\nu,\Theta,\mathbb{E}[{P}_h],\mathbb{L}}(y)\right\rVert_{L^1}\label{error-bestnet}\tag{e2}\\
\stackrel{\text{}}{+}&  \left\lVert \mathcal{N}_{\nu,\Theta,\mathbb{E}[{P}_h],\mathbb{L}}(y)-\mathcal{N}_{\nu,\Theta,{P}_h,\mathbb{L}}(y)\right\rVert_{L^1}\label{error-disk-payoff}\tag{e3}\\
\stackrel{\text{}}{+}&  \left\lVert \mathcal{N}_{\nu,\Theta,{P}_h,\mathbb{L}}(y)-\mathcal{N}_{\nu,\theta_K,{P}_h,\mathbb{L}}(y)\right\rVert_{L^1}\label{error-opti}\tag{e4}\\
\stackrel{\text{}}{+}&  \left\lVert \mathcal{N}_{\nu,\theta_K,{P}_h,\mathbb{L}}(y)-\mathcal{N}_{\nu,\theta_K,{P}_h,{\mathbb{L}}_M}(y)\right\rVert_{L^1}.\label{error-disk-loss}\tag{e5}
\end{align}
Let us shortly explain the above individual error sources, whereat we will have further studies on  the mathematical background of the decomposition in \cref{background}. The first line \eqref{error-disk-sde} describes the error made through the discretization of the stochastic process, using a discretization scheme with step-width $h$. Error \eqref{error-bestnet} describes the error made through the approximation of this function with a neural network with fixed structure $\nu$, whereat here we assume to know the best respective weights $\Theta$ (best fit). The error arising in \eqref{error-disk-payoff} results from the statistical error of using an approximation of the expectation. Finally, the optimization algorithm using $K$ iteration steps leads to \eqref{error-opti} and the approximation of the loss function by $M$ samples per training step results in \eqref{error-disk-loss}.

Since, we assume a fixed neural network structure $\nu$, the computational complexity $C_{\mathcal{N}}$ of the training process for a neural network can be bounded by
\begin{align*}
C_{\mathcal{N}}&\le c \left( K \cdot \text{cost}_{\text{ training step}}\right)
\end{align*}
with a positive constant $c$, where each of the $K$ training steps consist of the computation of the training data, evaluating of the loss function and modifying the weights to minimize the loss function.

Now, we want to compare both methods mentioned in the introduction based on this error decomposition.  The first approach uses an estimation of the expectation (e.g., computed through multilevel Monte Carlo) as $P_h$, computed typically on deterministically selected input parameters (for a low dimensional case). The second approach, introduced by \cite{beck2018solving}, uses only single paths instead of an expectation estimation, as $P_h$, computed on randomly selected input parameter vectors.

At a first glance, if choosing a proper $h$, both approaches should have similar properties for the error sources \eqref{error-disk-sde}, \eqref{error-bestnet} and \eqref{error-opti}, but differ in \eqref{error-disk-payoff}, \eqref{error-disk-loss} and obviously in the computation cost. However, as we will see in \cref{unbiasedness}, error \eqref{error-disk-payoff} is negligible. 
Now, by interpreting \eqref{error-disk-loss} as an integral approximation problem using $M$ samples for the estimation, a possible description of the differences could be as follows: While we expect the first approach to have higher convergence order with respect to batch-size, it may suffer under the curse of dimensionality for high dimensional training sets. On the other hand, we expect the approach of  \cite{beck2018solving} to have a lower convergence order with respect to batch-size, but it could be computed without suffering from the curse of dimensionality. Furthermore, as explained in the introduction, the cost for the training data computed with the first approach is of order $\mathcal{O}(\epsilon^{-2}$, where the cost for a path is of order $\mathcal{O}(h^{-1})$.

Both approaches, using somewhat simplified assumptions, could lead to similar complexities presented in the following two lemmata for a low-dimensional case. For the first approach, studied in \cref{corollar2.3}, we assume the convergence order of the loss function with respect to $M$ (batch-size) to be one. For the second approach, studied in \cref{corollar2.2}, we assume the loss function's convergence order with respect to $M$ to be $1/2$. For both approaches, we will assume a convergence of $1/2$ with respect to training steps $K$, and we consider a network structure $\nu$ such that \eqref{error-bestnet} can be neglected.

\begin{lemma}\label{corollar2.3} 
Consider the training of a neural network, as described in \eqref{net1}, for the approximation of \eqref{epy} on a training set $Y\subset \mathbb{R}$. If there exist positive constants $c_1,c_2,c_3$ and $c_4$ such that the decomposition \eqref{error-disk-sde}-\eqref{error-disk-loss} of \eqref{single-mse} can be bounded by
\begin{align*}
\eqref{error-disk-sde} &\le \frac{1}{4}\epsilon,   & \mathbb{E}[\eqref{error-opti}] &\le c_3 K^{-1/2},& \mathbb{E}[\eqref{error-disk-loss}] &\le c_2 M^{-1},\\
 \eqref{error-bestnet} &\le \frac{1}{4}\epsilon,  &  \mathbb{V}[\eqref{error-opti}] &\le c_3^2 K^{-1}, & \mathbb{V}[\eqref{error-disk-loss}] &\le c_2^2 M^{-2}, \\
\eqref{error-disk-payoff} &=0,  & & 
\end{align*}

such that the computational complexity of the training is be bounded by
 \begin{align*}
C_{\mathcal{N}}&\le c_4  K  M \epsilon^{-2}.
\end{align*}
Then there exist positive constants $c_5,M$ and $K$ such that  the overall error \eqref{single-mse} can be bounded by $\epsilon^2$ with a computational complexity bound of
\begin{align*}
C_{\mathcal{N}}\le c_{5} \epsilon^{-5}.
\end{align*}
\end{lemma}

\begin{proof}
See \cref{proof}.
\end{proof}

\begin{lemma}\label{corollar2.2} 
Consider the training of a neural network, as described in \eqref{net1}, for the approximation of \eqref{epy} on a training set $Y\subset \mathbb{R}$. If there exist positive constants $c_1,c_2,c_3$ and $c_4$ such that the decomposition \eqref{error-disk-sde}-\eqref{error-disk-loss} of \eqref{single-mse} can be bounded by
\begin{align*}
\eqref{error-disk-sde} &\le c_1 h,   & \mathbb{E}[\eqref{error-opti}] &\le c_3 K^{-1/2},& \mathbb{E}[\eqref{error-disk-loss}] &\le c_2 M^{-1/2},\\
 \eqref{error-bestnet} &\le \frac{1}{4}\epsilon,  &  \mathbb{V}[\eqref{error-opti}] &\le c_3^2 K^{-1}, & \mathbb{V}[\eqref{error-disk-loss}] &\le c_2^2 M^{-1}, \\
\eqref{error-disk-payoff} &=0,  & & 
\end{align*}

such that the computational complexity of the training is be bounded by
 \begin{align*}
C_{\mathcal{N}}&\le c_4 K M  h^{-1}.
\end{align*}
Then there exist positive constants $c_5, h,M$ and $K$ such that  the overall error \eqref{single-mse} can be bounded by $\epsilon^2$ with a computational complexity bound of
\begin{align*}
C_{\mathcal{N}}\le c_{5} \epsilon^{-5}.
\end{align*}
\end{lemma}

\begin{proof}
See \cref{proof}.
\end{proof}

If the assumptions are reasonable, improved numerical integration techniques such as sparse grid integration for the first approach, see, e.g., \cite{gerstner1998numerical}, or variance reduction,  for the second approach, see, e.g., \cite{glasserman}, should reduce the overall complexity. Further studies on the implications of variance reduction techniques for the second approach will follow in \cref{numerical-results}.

\subsection{Multilevel Monte Carlo training}\label{multilevel-error-sec}

The following section will contain the error decomposition, the main complexity theorem, and a numerical experiment for the multilevel Monte Carlo training approach. 

\subsubsection{error decomposition}
In this section, we will study the error sources of the multilevel training approach.

Since, the approach consists in training of individual neural networks for each level, we aim to find suitable weights on each level, such that the networks are approximations of the expectations of \eqref{levelmc2}. We aim to bound
\begin{align*}
\left\rVert  \mathbb{E}[\hat{Y}_l(y)]- \mathcal{N}_{\nu,\theta^l_{K_l},\hat{Y}_l,{\mathbb{L}}_{M_l}}(y) \right\rVert_{L^p},
\end{align*}
for each level $l=0,\dots,L$ using $K_l$ training steps and loss functions discretized by using batch-size $M_l$. Again, we assume all networks to have the same fixed structure $\nu$. 

Then, the final multilevel approximation $\hat{\mathcal{N}}$ will be given by
\begin{align}
\hat{\mathcal{N}}: y \mapsto \mathcal{N}_{\nu,\theta^0_{K_0},\hat{Y}_0,\hat{\mathbb{L}}_{M_0}}(y) + \sum\limits_{l=1}^L \mathcal{N}_{\nu,\theta^l_{K_l},\hat{Y}_l,\hat{\mathbb{L}}_{M_l}}(y),\label{multilevel-definition}
\end{align}
whereat it should satisfy
\begin{align}
\mathbb{E}\left[\left\lVert \bar{P}(y)-\hat{\mathcal{N}}(y)\right\rVert_{L^1}^2\right]< \epsilon^2,\label{error-bound-multi}
\end{align}
for a required error $\epsilon>0$. Using the standard variance expansion, the left hand side can be expanded to
\begin{align}
&
\mathbb{V}\left[\left\lVert \bar{P}(y)-\hat{\mathcal{N}}(y)\right\rVert_{L^1}\right]+
\mathbb{E}\left[\left\lVert \bar{P}(y)-\hat{\mathcal{N}}(y)\right\rVert_{L^1}\right]^2.\label{decomposition}
\end{align}
Again, we specify the error sources by decomposing the inner term of \eqref{error-bound-multi} by
\begin{align}
\left\lVert \bar{P}(y) - \hat{\mathcal{N}}(y)\right\rVert_{L^1}
\stackrel{\text{}}{\le}& \left\lVert \bar{P}(y)-\mathbb{E}[\hat{P}_{h_L}(y)]\right\rVert_{L^1}\label{error-a}\tag{E1}\\
\stackrel{\text{}}{+}& \sum_{l=0}^L \left\lVert \mathbb{E}[\hat{Y}_l(y)]-\mathcal{N}_{\nu,\Theta^l,\mathbb{E}[\hat{Y}_l],{\mathbb{L}}}(y)\right\rVert_{L^1}\label{error-b}\tag{E2}\\
+&\sum_{l=0}^L   \left\lVert  \mathcal{N}_{\nu,\Theta^l,\mathbb{E}[\hat{Y}_l],{\mathbb{L}}}(y)- \mathcal{N}_{\nu,\Theta^l,\hat{Y}_l,{\mathbb{L}}}(y)\right\rVert_{L^1}\label{error-c}\tag{E3}\\
+&\sum_{l=0}^L  \left\lVert \mathcal{N}_{\nu,\Theta^l,\hat{Y}_l,{\mathbb{L}}}(y)-\mathcal{N}_{\nu,\theta^l_{K_l},\hat{Y}_l,{\mathbb{L}}}(y))\right\rVert_{L^1}\label{error-d}\tag{E4}\\
+&\sum_{l=0}^L  \left\lVert \mathcal{N}_{\nu,\theta^l_{K_l},\hat{Y}_l,{\mathbb{L}}}(y)-\mathcal{N}_{\nu,\theta^l_{K_l},\hat{Y}_l,{\mathbb{L}}_{M_l}}(y)\right\rVert_{L^1}\label{error-e}\tag{E5}.
\end{align}

The explanations of the error sources are analogous to those of \eqref{error-disk-sde} to \eqref{error-disk-loss}, whereas from \eqref{error-b} to \eqref{error-e} we use the sum of the $L+1$ errors.

Again by fixing the net structures, the computational complexity $C_{\mathcal{N}_l}$ of the training for each network $\mathcal{N}_l:=\mathcal{N}_{\nu,\theta^l_{K_l},\hat{Y}_l,{\mathbb{L}}_{M_l}}$ is bounded by
\begin{align}
C_{\mathcal{N}_l}&\le c \left( K_l \cdot \text{cost}_{\text{training step on level}}\right),\label{cost-multi}
\end{align}
with a positive constant $c$. Each of the $K_l$ training steps consist of the computation of the $M_l$ training data of type $\hat{Y}_l$, the calculation of the loss function $\hat{\mathbb{L}}_l$ and the modification of the weights to minimize the loss function.

\subsubsection{complexity theorem}\label{complexity}

In this subsection, we present the main complexity theorem. Since further studies on the error sources are required, we will word some assumptions quite generally. The idea of theoretical complexity reduction can be described as follows. The multilevel approach uses different time-steps for the discretization to obtain a complexity reduction. Under certain circumstances, this induces different amounts of random samples on each level, especially fewer samples on the finer discretizations. The idea can be transformed into multilevel training, which - when using finer path discretizations - we expect to result in lower needed batch-sizes. As introduced by \cite{giles2008multilevel} the variance of the \textit{level estimators} needs to be bounded with an order of $\mathcal{O}(h_l^\beta)$, for $\beta>0$. For e.g. a European call option Giles, Debrabant and R\"ossler  prove $\beta=2$ in \cite{giles2013numerical}.

\begin{theorem}\label{complexity-thm}
Consider a multilevel training process as described in \eqref{multilevel-definition}, for the approximation of \eqref{epy} on a specific training set $Y$.
If there exist positive constants $\alpha\ge 1/2, \beta, \gamma, \eta, c_1,c_2,c_3$ and $c_4$ such that the decomposition \eqref{error-a}-\eqref{error-e} of \eqref{error-bound-multi} can be bounded by
\begin{align}
\left\lVert \bar{P}(y)-\mathbb{E}[\hat{P}_{h_l}(y)]\right\rVert_{L^1} &\le c_1  h_l, \tag{A1}\label{a1}\\
\left\lVert \mathbb{E}[\hat{Y}_l(y)]-\mathcal{N}_{\nu,\Theta^l,\mathbb{E}[\hat{Y}_l],{\mathbb{L}}}(y)\right\rVert_{L^1}&\le \frac{1}{\sqrt{32}}\epsilon,  \tag{A2}\label{a2}\\
 \left\lVert  \mathcal{N}_{\nu,\Theta^l,\mathbb{E}[\hat{Y}_l],{\mathbb{L}}}(y)- \mathcal{N}_{\nu,\Theta^l,\hat{Y}_l,{\mathbb{L}}}(y)\right\rVert_{L^1} &=0, \tag{A3}\label{a3}\\
\mathbb{E}\left[ \left\lVert \mathcal{N}_{\nu,\Theta^l,\hat{Y}_l,{\mathbb{L}}}(y)-\mathcal{N}_{\nu,\theta^l_{K_l},\hat{Y}_l,{\mathbb{L}}}(y))\right\rVert_{L^1}\right]&\le c_2  h_l^{\alpha\gamma} K_l^{-1/2}\tag{A4 i},\label{a4i}\\
\mathbb{V}\left[ \left\lVert \mathcal{N}_{\nu,\Theta^l,\hat{Y}_l,{\mathbb{L}}}(y)-\mathcal{N}_{\nu,\theta^l_{K_l},\hat{Y}_l,{\mathbb{L}}}(y))\right\rVert_{L^1}\right]&\le c_2^2  h_l^{2\alpha\gamma} K_l^{-1}\tag{A4 ii},\label{a4ii}\\
\mathbb{E}\left[\left\lVert \mathcal{N}_{\nu,\theta^l_{K_l},\hat{Y}_l,{\mathbb{L}}}(y)-\mathcal{N}_{\nu,\theta^l_{K_l},\hat{Y}_l,{\mathbb{L}}_{M_l}}(y)\right\rVert_{L^1}\right] &\le c_3  h_l^{\beta/2} \rho_l^{\eta} M_l^{-1/2}\tag{A5 i},\label{a5i}\\
\mathbb{V}\left[\left\lVert \mathcal{N}_{\nu,\theta^l_{K_l},\hat{Y}_l,{\mathbb{L}}}(y)-\mathcal{N}_{\nu,\theta^l_{K_l},\hat{Y}_l,{\mathbb{L}}_{M_l}}(y)\right\rVert_{L^1}\right] &\le c_3^2  h_l^{\beta} \rho_l^{2\eta} M_l^{-1}\tag{A5 ii},\label{a5ii}
\end{align}
with  $\rho_l=K_l^{-1/2}$, such that the computational complexity of the training for each net is be bounded by 
\begin{align}
C_l\le c_4 h_l^{-1}  M_l K_l.\label{complexity-bound}
\end{align}

Then, there exists a positive constants $c_5$ such that for any $\epsilon<e^{-1}$, there are values $L,M_l$ and $K_l$ for which
\begin{align}
\mathbb{E}\left[\left\lVert \bar{P}(y)-\hat{\mathcal{N}}(y)\right\rVert_{L^1}^2\right]\label{net-mse2}
\end{align}
 can be bounded by $\epsilon^2$, with a computational complexity $C_{\mathcal{N}}$ with bound
\begin{align*}
C_{\mathcal{N}} \le c_{5}
\begin{cases}
\epsilon^{-3.0}, & \text{ for } \eta=0.5,\hspace{5pt} \gamma=2  ,\beta = 2, \alpha=1, \\
\epsilon^{-3.0} \left|\log \epsilon\right|^4, & \text{ for } \eta=0.5,\hspace{5pt} \gamma=1 ,\beta = 1, \alpha=1, \\
\epsilon^{-3.5} \left|\log \epsilon\right|^4, & \text{ for } \eta=0.25, \gamma=1 ,\beta = 1, \alpha=1, \\
 \epsilon^{-4}\left|\log \epsilon\right|^5,& \text{ for } \eta=0,\hspace{13pt} \gamma=0, \beta = 1, \alpha=1, \\
  \epsilon^{-5}\left|\log \epsilon\right|^4,& \text{ for } \eta=0,\hspace{13pt} \gamma=0, \beta = 0, \alpha=1, \\
    \epsilon^{-6}\left|\log \epsilon\right|^4,& \text{ for } \eta=0,\hspace{13pt} \gamma=0, \beta = 0, \alpha=1/2. \\
\end{cases} 
\end{align*}
\end{theorem}

\begin{proof}
See \cref{proof}
\end{proof}
\begin{remark}
By assuming $\eta=\gamma=0$ and $ \beta=\alpha=1$, we focus on the level effect with respect to the batch-size, as, e.g., studied in the introductory example.
\end{remark}

\subsubsection{Numerical results}\label{numerical-results}
We will now present an extension of the example of \cref{sec-example} by considering the level effect with respect to training steps.

We extend the example of \cref{sec-example} by including the level effect with respect to training steps, as introduced in \cref{complexity-thm}.

Using the identical model and network parameters, we will use the number of training steps and batch-sizes, as stated in \cref{multi-steps}.

\begin{table}
 \small
\centering
  \begin{tabular}{c|c|c|c|c|c|c|c|c}
  \hline
  level $l$ & $0$ & $1$ & $2$ & $3$  & $4$ & $5$  & $6$  & $7$  \\  
\hline
$M_l$ & $1200000$ &$64000$ &$32000$&$16000$&$8000$&$4000$&$2000$& $1000$ \\
$K_l$ & $150000$ & $20000$ &$19000$&$18000$&$15000$&$14000$&$13000$& $11000$ \\
\hline 

 \end{tabular}
 \caption{Batch-sizes $M_l$ and training steps $K_l$ for the training of the specific level nets $\mathcal{N}_l$ for $l=1,\dots,8$.}
  \label{multi-steps}
\end{table}

The mean and standard deviation of the $L^\infty$ for this modified training are presented in \cref{multi-training-3}.

\begin{table}
 \small
\centering
  \begin{tabular}{c|c}
  \hline
   single net: mean error (time)  & multilevel: mean error (time)     \\
 \hline
  	  $0.0119 $ $(26.66h)$  & $0.0111$ $(9.66h)$  \\
 \hline
 \end{tabular}
 \caption{Mean of the $L^\infty$-error for the single-level algorithm using the Milstein discretization and the multilevel algorithm using respective batch-sizes as presented in \cref{multi-steps}. All simulations were repeated $10$ times. The brackets show needed computation time using a Nvidia K80 GPU.  }
  \label{multi-training-3}
\end{table}

The multilevel training results using the stated amount of levels, batch-sizes, and training steps are given in \cref{multi-training-3}.

Compared to \cref{results2}, we see a further computational reduction while satisfying the required error-bound.

\section{Extensions}
In this subsection, we want to discuss some promising possible extensions beyond this work's scope except the first one. 

\subsection{Optimal $K_l$}

As explained in the introduction, we could use a second level effect by using a possible connection between the variance of the \textit{level estimators'} variance and training steps needed for the respective network. In the example above, we explicitly used the same parameters of the single-level for the multilevel approach. This method implied a relatively straightforward procedure to implement the multilevel approach and still had an advantage in computation time. However, we observed that for the finest level, much fewer training steps would be sufficient. Therefore, it could lead to further savings using individual decay rates, initial learning rates, or training steps. Further studies are made in \cref{complexity}.

\subsection{Optimal $N_l$}\label{nl-extension}

We only used single paths ($N_l=1$) to compute the training data for the computations above. We discussed the advantages and disadvantages of explicitly computed prices in \cref{corollar2.2} and \cref{corollar2.3}. Furthermore, we will see in \cref{unbiasedness} that this approach does not lead to a bias for the final trained neural network under certain assumptions. Nevertheless, if we study \eqref{a4i} to \eqref{a5ii} we see that the estimator's variance could affect the needed batch-size $M_l$ and the needed training steps $K_l$. Therefore, we believe that further studies on $N_l$ could lead to efficiency improvements. For example, for levels with a high estimators' variance, increasing $N_l$ could be reasonable.

\subsection{Optimal $H$}

In our work so far, we have not specified $H$, which is the factor by which the time-step is refined. Again, we refer to \cite{giles2008multilevel} for some further explanations on $H$ for the multilevel Monte Carlo estimator. For the multilevel Monte Carlo simulation, further studies on $H$ could lead to further efficiency improvements for the multilevel training.

\subsection{Individual structure $\nu_l$}

In this work, we fixed the net structure $\nu$ for each neural net, and for simplicity, we assumed the approximation error to be negligible. Nevertheless, we see the potential for further studies on individual structures $\nu_l$, which could further improve efficiency.

\section{Conclusion}\label{Conclusion}

In this work, we combined the ideas of a single-paths deep learning approximation with the multilevel Monte Carlo path simulation concept. We showed that the resulting multilevel Monte Carlo training approach could reduce the complexity of the training process. 

 Deep learning algorithms have become very popular in recent years. However, there are not many rigorous mathematical convergence results for the different error sources. The decomposition into three parts, the approximation error, the generalization error, and the optimization error, and the first overall error analysis was made by \cite{beck2019full}. However, their convergence speed analysis is far from optimal and suffers from the curse of dimensionality. Hence, we worded the main theorem quite generally. Nevertheless, several challenging areas for further research arise in this work.
 
Therefore, the first is a more in-depth theoretical analysis of the convergence speed with respect to each error source. 

The second is the individual network structure. In this work, we used a fixed network structure for each net. Using distinct network structures could lead to further numerical savings. Furthermore, we ignore a possible third level-effect arising if the network structures could be linked to the level. 

A further research area concerns the heuristic of the multilevel algorithm. Like the heuristic for the batch-sizes, it would be desirable to have an algorithm to obtain an optimal amount of training steps for each net. 

Finally, it may be possible to further reduce the complexity by using more than just one path to simulate the training data. To achieve that, we must ensure a better understanding of the introduced level-effect parameter. 

\bibliography{references}
\newpage
\appendix

The appendix's proceeding will be as follows: First, we give a definition of a neural network and have a closer look at the mathematical background of the error sources \eqref{error-a} to \eqref{error-e}. 
Then, we present some numerical results aiming to support some of the discussed assumptions. Finally, we will provide the missing proof.

\section{Mathematical background}\label{background} In this subsection we will give some mathematical background and references. For this, we first give a short definition of neural networks used in this work. Then, we discuss the error sources \eqref{error-a} to \eqref{error-e}. Finally, we present an example of an \textit{level estimator'} code.

\subsection{Neural network defintion and optimization}\label{neuralnet}

 Let $\mathcal{L}_d : \mathbb{R}^d \rightarrow \mathbb{R}^d$ be the function which satisfies for every $y=(y_1,y_2,\dots,y_d) \in \mathbb{R}^d$ that
\begin{align*}
\mathcal{L}_d(y) = \left( \frac{\exp(y_1)}{\exp(y_1)+1},\frac{\exp(y_2)}{\exp(y_2)+1}, \dots, \frac{\exp(y_d)}{\exp(y_d)+1} \right),
\end{align*}
for every $k,l \in \mathbb{N}, v \in \mathbb{N}_0, \theta=(\theta_1,\dots, \theta_\nu) \in \mathbb{R}^\nu$ with $v + l(k+1) \le \nu$, let $A_{k,l}^{\theta,v}: \mathbb{R}^k \rightarrow \reel{R}^l$ be the function which satisfies for every $x=(x_1,\dots, x_k) \in \reel{R}^k$ that
\begin{align*}
A_{k,l}^{\theta,v}(x)=
 \begin{pmatrix}
\theta_{v+1} & \theta_{v+2} & \hdots & \theta_{v+k}  \\
\theta_{v+k+1} & \theta_{\nu+k+2} & \hdots & \theta_{v+2k}\\
\theta_{v+2k+1} & \theta_{v+2k+2} & \hdots & \theta_{v+3k}\\
\vdots & \vdots & \vdots & \vdots \\
\theta_{v+(l-1)k+1} & \theta_{v+(l-1)k+2} & \hdots & \theta_{v+lk}
\end{pmatrix}
 \begin{pmatrix}
x_1\\
x_2\\
x_3\\
\vdots \\
x_k
\end{pmatrix}
+
 \begin{pmatrix}
\theta_{v+kl+1}\\
\theta_{v+kl+2}\\
\theta_{v+kl+3}\\
\vdots \\
\theta_{v+kl+l}
\end{pmatrix},
\end{align*}
let $s \in \{ 3,4,5,6,\dots \}$, assume that $(s-1)d(d+1)+d+1 \le \nu$ and let $\mathcal{N}_{\nu,\theta}: \mathbb{R}^d \rightarrow \mathbb{R}$ be the function which satisfies for every $y\in \mathbb{R}^d$ that

\begin{align*}
\mathcal{N}_{\nu,\theta} (y) = \left(A^{\theta,(s-1)d(d+1)}_{d,1} \circ \mathcal{L}_d \circ A^{\theta,(s-2)d(d+1)}_{d,d} \circ \dots \circ \mathcal{L}_d \circ
A^{\theta,d(d+1)}_{d,1} \circ \mathcal{L}_d A^{\theta,0}_{d,d} \right) (y).
\end{align*}
The function $\mathcal{N}_{\nu,\theta}$ describes an artificial neural network with $s+1$ layers and standard logistic functions as activation functions. Fixing the network structure $\nu$ and assuming to know the best fit weights $\Theta\in \mathbb{R}^\nu$ leads to the approximation error \eqref{error-b}:
\begin{align}
\left\lVert \bar{P}(y) - \mathcal{N}_{\nu,\Theta}(y) \right\rVert_{L^p}. \label{assumption1}
\end{align} 
For an overview over the approximation error, we refer to \cite{higham2019deep} and to \cite{barron1994approximation,hornik1991approximation,hornik1989multilayer,hornik1990universal} for further studies.

Now, consider a neural net $\mathcal{N}_{\nu,\theta}$, as defined in \eqref{assumption1}, but without knowing the best-fit-weights $\Theta$. Aiming to train the network such that it approximates $\bar{P}$, we are interested in minimizing the loss function
\begin{align}
\mathbb{L}:\theta \mapsto \left\lVert \bar{P}(y) - \mathcal{N}_{\nu,\theta}(y) \right\rVert_{L^p}.\label{conti-loss}
\end{align}

The Taylor series expansion gives
\begin{align*}
\mathbb{L}(\theta + \Delta \theta)= \mathbb{L}(\theta) + \sum_{r=1}^\nu \frac{\partial \mathbb{L}(\theta)}{\theta_r} \Delta \theta,
\end{align*}
where $\partial \mathbb{L}(\theta)/\theta_r$ denotes  the  partial  derivative  of  the  loss  function  with respect  to  the $r$-th parameter. Generally, only being  an approximation for a small step $\Delta \theta$, we limit the step in that direction by $\eta$, leading to series of weights
\begin{align}
\theta_{i+1} :=  \theta_i - \eta \nabla \mathbb{L}(\theta_i),\label{heuristic-opti}
\end{align}  
for $i=0,\dots,K-1$ and with the so-called learning rate $\eta$. We call an iteration of this form a training step. All in all this leads to the optimization errors \eqref{error-opti} or \eqref{error-d}.

We complete the neural network training definition by defining the following series of functions
\begin{align}
\mathcal{N}_{\nu,\theta_i,P,\mathbb{L}}: \mathbb{R}^d\rightarrow \mathbb{R}\label{net-series},
\end{align}
with $i=1,\dots,K$, initial weights $\theta_0$ and loss function $\mathbb{L}$, which is evaluated with training data $P$.

In this work, we use stochastic gradient descent optimization and backpropagation, see e.g., \cite{higham2019deep}, for an overview. For studies on the optimization error, see, e.g.,  \cite{bach2013non,bercu2011generic,chau2019stochastic,fehrman2020convergence,jentzen2021strong}.

\subsection{Discretization error}

We shortly introduce the parameters $\alpha$ and $\beta$. The (weak) convergence, for a fixed parameter vector $y\in Y$, is defined by
\begin{align}
 \mathbb{E}[P-P_h]   \leq c  h^\alpha,\label{alpha}
\end{align}
with a positive constant $c$. It is well known, provided certain assumptions are satisfied that
both the Euler scheme and the Milstein scheme \eqref{milstein} converge with $\alpha=1$ for Lipschitz-continuous payoff $P$ only depending on the time of maturity, see, e.g., the monograph \cite{kloeden2012numerical}. 
The second parameter $\beta$ is used for bounding the estimators' variance, as explained in \cref{complexity-thm}.
This parameter is, e.g., the focus of the article \cite{giles2013numerical}. The authors prove different constants for $\beta$ for a different type of option under certain SDE conditions. 


\subsection{Generalization error}\label{unbiasedness}
This section will give some background for a more in-depth understanding of the loss functions' discretization error. As explained above, we discretize the loss function \eqref{conti-loss} in each training step by a finite set $y_i\in Y$, with $i=1,\dots,M$. 

The usage of the approximated loss function leads to the generalization error \eqref{error-disk-loss}. For the multilevel approach, we discretize the loss function at randomly selected inputs $y_i \in Y$. 

Consider the loss function \eqref{conti-loss}. First, we transform both arising integrals to the unit cube. For simplicity, we assume the training set $Y$ to be of the form $[a,b]^d$, without loss of generality. For the transformation of the norm-part, we will use the linear transformation $\mathcal{U}:[a,b]^d\rightarrow [0,1]^d$. For the expectation-integral, we will use the inversion method, see e.g. \cite{glasserman}, e.g. for the cumulative standard normal distribution function $\Phi(x): \mathbb{R}^H \rightarrow [0,1]^H$, leading to the following lemma.
\begin{lemma}
Let $Y=[a,b]^d$ and consider $P_{h}$ to be multivariate standard normal distributed for each $y\in Y$. Then, by using the linear transformation $\mathcal{U}:Y\rightarrow [0,1]^d=:U$ and the cumulative standard normal distribution function $\Phi(x): \mathbb{R}^H \rightarrow [0,1]^H$, we have
\begin{align*}
\left\lVert \mathbb{E}[{P}_h(y)] - {\mathcal{N}}_{\nu,\theta,\mathbb{E}[{P}_h],\mathbb{L}_Y}(y) \right\rVert_{L^p(Y)}&=\\&(b-a)^{-d/p}\left\lVert \mathbb{E}^*[\hat{P}_h(u)]- \hat{\mathcal{N}}_{\nu,\theta,\mathbb{E}[\hat{P}_h],\mathbb{L}_U}(u) \right\rVert_{L^p(U)} ,
\end{align*}
with $\hat{P}_h(u): u \mapsto P_h(\mathcal{U}^{-1}(u))$, $\hat{\mathcal{N}}_{\nu,\theta,\mathbb{E}[\widehat{P}^h],\mathbb{L}_U}:  u \mapsto \mathcal{N}_{\nu,\theta,\mathbb{E}^*[{P}^h],\mathbb{L}_Y}(\mathcal{U}^{-1}(u))$, and the expectation $\mathbb{E}^*$ defined on the unit cube.
\end{lemma}
\begin{proof}
\begin{align*}
&\left\lVert \mathbb{E}[P_h(y)] - \mathcal{N}_{\nu,\theta,\mathbb{E}[P_h],\mathbb{L}_Y}(y) \right\rVert_{L^p}\\
=&\left( \int_{[a,b]^d} \left\vert \mathbb{E}[P_h(y)]-\mathcal{N}_{\nu,\theta,\mathbb{E}[P^h],\mathbb{L}_Y}(y) \right\vert^{p}\diff y \right)^{1/p}\\
=& \left( (b-a)^d\int_{[0,1]^d} \left\vert \mathbb{E}[P_h(\mathcal{U}^{-1}(u))] - \mathcal{N}_{\nu,\theta,\mathbb{E}[P_h],\mathbb{L}_Y}(\mathcal{U}^{-1}(u)) \right\vert^p \diff u\right)^{1/p}\\
=&(b-a)^{d/p} \left\lVert \mathbb{E}^*[\widehat{P}_h(u)] - \widehat{\mathcal{N}}_{\nu,\theta,\mathbb{E}[\widehat{P}^h],\mathbb{L}_U}(u) \right\rVert_{L^p}.
\end{align*}
\end{proof}
I.e. we can study the neural net training on the unit cube. 
\begin{corollary}\label{norm1}
Consider $Y=[0,1]^d=:U$ and $P_h$ to be uniformly distributed. We have the following inequality:
\begin{align}
\left\lVert \mathbb{E}[{P}_h(u)] - {\mathcal{N}}(u) \right\rVert_{L^1}^2\label{norm1gl}
\le&\left\lVert \mathbb{E}[{P}_h(u)] - {\mathcal{N}}(u) \right\rVert_{L^2}^2
\end{align}
and the difference is given by
\begin{align}
\int_U \mathbb{V}\left[ \mathbb{E}[{P}_h(u)] - {\mathcal{N}}(u) \right]\diff u.\label{diff1}
\end{align}
\end{corollary}
\begin{proof}
\eqref{norm1gl} is given through H{\"o}lders' ineqaulity. \eqref{diff1} is given by
\begin{align*}
\int_U&\mathbb{V}^*\left[ \mathbb{E}[\widehat{P}_h(u)] - {\mathcal{N}}(u) \right]\diff u\\
&=\int_{U} \left(\mathbb{E}[{P}_h(u)] - \widehat{\mathcal{N}}(u)\right)^2 \diff u - \left(\int_{U} \mathbb{E}[{P}^h(u)] - {\mathcal{N}}(u)\diff u\right)^2.
\end{align*}
\end{proof}

In practice, we will be interested in approximating the right-hand side of \eqref{norm1gl}. The standard Monte Carlo approximation leads to the following corollary.

\begin{corollary}\label{mc1}
For $p=2$, the right-hand side of \eqref{norm1gl} can be estimated by the Monte Carlo estimator
\begin{align*}
\frac{1}{M}\sum\limits_{i=1}^M \left(\left( \frac{1}{N}\sum\limits_{j=1}^N P(u^{(2)}_{i,j},u^{(1)}_i)\right) - \mathcal{N}(u^{(1)}_i) \right)^2,
\end{align*}
with $u^{(1)}_{i} \in [0,1]^{d}$ for $i=1,\dots,M$ and $u^{(2)}_{i,j} \in [0,1]^{H}$ for $j=1,\dots N$ and $i=1,\dots,M$ both i.i.d. uniformly distributed. This estimator leads to the \textit{generalization error}.
\end{corollary}

For, e.g., $d=H=1$, this estimator needs $M+MN$ random samples for the training data for each training step. The introduced Monte Carlo estimator relies on the choice of both $M$ and $N$. The attempt to avoid this choice leads to the following considerations.
\begin{lemma}
For the right-hand side of \eqref{norm1gl} the following inequality holds:

\begin{align}
\left\lVert \mathbb{E}[{P}_h(u)] - {\mathcal{N}}(u) \right\rVert_{L^2}^2
\le  \int_{[0,1]^{d+H}}  \left( {P}_h(u) - \mathcal{N}(u)\right)^2 \diff u\label{norm2gl}
\end{align}
and the difference is given by
\begin{align}
\left\lVert \mathbb{V}\left[{P}_h(u) \right] \right\rVert_{L^1}\label{diff2}.
\end{align}

\end{lemma}

\begin{proof}
Let $u^{(1)}\in [0,1]^d$ and $u^{(2)}\in [0,1]^H$. Then,
\begin{align*}
&\left\lVert \mathbb{E}[{P}_h(u)] - {\mathcal{N}}(u) \right\rVert_{L^2}^2\\
=&  \int_{[0,1]^d} \left( \int_{[0,1]^H}{P}^h(u^{(2)},u^{(1)})\diff u^{(2)} - {\mathcal{N}}(u^{(1)}) \right)^2\diff u^{(1)}\\
&\le  \int_{[0,1]^d}  \int_{[0,1]^H}\left( {P}^h(u^{(2)},u^{(1)}) - {\mathcal{N}}(u^{(1)})\right)^2\diff u^{(2)} \diff u^{(1)}
\end{align*}
The difference follows from above with 
\begin{align*}
&=  \int_{[0,1]^d}  \int_{[0,1]^H}\left( {P}^h(u^{(2)},u^{(1)}) - {\mathcal{N}}(u^{(1)})\right)^2\diff u^{(2)} \diff u^{(1)}\\
&- \int_{[0,1]^d} \left( \int_{[0,1]^H}{P}^h(u^{(2)},u^{(1)})\diff u^{(2)}  - \int_{[0,1]^H}{\mathcal{N}}(u^{(1)})\diff u^{(2)} \right)^2\diff u^{(1)}\\
&= \int_{[0,1]^d} \mathbb{V}\left[{P}^h(u) - \mathcal{N}(u)\right] \diff u.
\end{align*}
\end{proof}

The Monte Carlo approximation of the right-hand side of \eqref{norm2gl} leads to the following corollary.

\begin{corollary}\label{mc2}
The loss function $\mathbb{L}^2_{2^*}$ defined on the right-hand side of \eqref{norm2gl} can be estimated by the Monte Carlo estimator
\begin{align*}
\widehat{L}_{M}=\frac{1}{M}\sum\limits_{i=1}^M \left( [\widehat{P}(u_{2,i},u_{1,i})] - \mathcal{N}_{\nu,\theta,\widehat{P}}(u_{1,i}) \right)^2,
\end{align*}
with $u_{1,i} \in [0,1]^{d}$ and $u_{2,i} \in [0,1]^{H}$, for $i=1,\dots,M$ independent uniformly distributed.
\end{corollary}

For, e.g., $d=H=1$, this leads to $2M$ needed random numbers per iteration step. 

Furthermore, from, e.g., proposition 2.2 of \cite{beck2019full}, we know that under certain assumptions, there exists a a neural network $N:[0,1]^d\rightarrow \mathbb{R}$ and a unique continuous function $f$ such that
\begin{align}
\inf\limits_{f  \in \mathcal{C}([0,1]^d,\mathbb{R})} \int_{[0,1]^{d+H}}  \left( {P}^h(u) - f(u)\right)^2 \diff u= \int_{[0,1]^{d+H}}  \left( {P}^h(u) - N(u)\right)^2 \diff u \label{jen1}
\end{align}
and it holds for every $u\in [0,1]^d$ that
\begin{align}
N(u)=\int\limits_{[0,1]^H} {P}^h(u) \diff u = \mathbb{E}\left[{P}^h(u)\right]. \label{jen2}
\end{align}
\begin{remark}
In other words, the function minimizing the right-hand side of \eqref{norm2gl} could as well minimize the right-hand side of \eqref{norm1gl}. I.e., for the training process of a neural network, both Monte Carlo estimators of \cref{mc2} and \cref{mc2} lead to unbiased results. With respect to the error assumptions, this justifies \eqref{a3}. However, one should keep in mind that the Monte Carlo estimator of \cref{mc1} could be more efficient, even though, depending on the choice of $N$ and $M$. The apparent reason for this is that even though the norm choice does not apply a bias, e.g., in \eqref{a3}, its respective Monte Carlo estimators' quality will affect the loss functions' error, e.g., in \eqref{a5i}.
\end{remark}

For further studies on the generalization error, see, e.g., \cite{berner2020analysis,poggio2002mathematical,gyorfi2006distribution,shalev2014understanding,geer2000empirical}.

We will study a single-level training process for a European option using a geometric Brownian motion in the second subsection. While neglecting the training step-size, we will focus on the numerical properties of the loss functions' error.


\subsection{Numerical results: batch-size convergence and variance reduction}

Even though we know that the standard Monte Carlo simulation converges with order $1/2$, it is not entirely clear whether this property applies to the neural network training, as assumed, e.g., in \eqref{a5i}. Hence, we present some results for increasing batch-sizes for the GBM and a European call option for a small fixed training set, e.g., $s_0\in [100,104]$. Using this small training set, the closed solution as training data, and a vast amount of training steps, we suppose the batch-size to be the crucial factor for the error. A first result can be found in \cref{result}. If not further specified, we use the same network and training parameters as in the examples above.

\begin{table}
 \small
\centering
  \begin{tabular}{c|c|c|c}
  \hline
batch-size  & mean error &  standard deviation  & mean reduction  \\
 \hline
 $1k$ 	  &	$0.0263$  &  $0.0102 $ &	\\
 $4k$   & $0.0139  $ & $ 0.0042 $ & $ 0.5285 $ \\
 $16k$   & $0.0071 $ & $0.0022 $ & $0.5096 $\\
 $64k$   & $ 0.00321 $ & $0.0010 $ & $ 0.4500$\\
 \hline
 \end{tabular}
 \caption{Mean and standard deviation of the $L^\infty$-error with respect to increasing batch-size. Training set is $[100,104]$ and the training is repeated $10$ times. The last column describes the error reduction.}
  \label{result}
\end{table}

We observe a convergence order of approximately $1/2$ for an increasing batch-size with respect to the mean of the sampled $L^\infty$ error.

As well known, our specific choice of the payoff leads to a particular standard deviation of the standard Monte Carlo estimator. Furthermore, it is well known that this standard deviation is responsible for the Monte Carlo error. It could be reduced by, e.g., variance reduction techniques, see, e.g., \cite{glasserman}. For our example, we will use some importance sampling by only sampling paths that stay above the strike price. Since we know that the variance is increasing monotonously for this specific example, we evaluated the variance reduction factor to be included in the interval $[0.34179,0.39067]$ for our training set's initial values. Now, by training a neural network with training data computed with the variance reduced Monte Carlo estimator, we obtain results as can be seen in \cref{result-standard-vs-oss}.

\begin{table}[H]
 \small
\centering
  \begin{tabular}{c|c|c|c}
  \hline
batch-size  & mean error standard net & mean error OSS net & mean reduction \\
 \hline
 $1k$   &	$ 0.0263 $  &  $ 0.0106  $ & $0.3869 $  	\\
 \hline
 \end{tabular}
 \caption{Mean of the $L^\infty$-error for both the neural networks. The calculated mean reduction is given in the last column. The training was computed on the interval $[100,104]$ and repeated $10$ times.}
  \label{result-standard-vs-oss}
\end{table}

Comparing the results' mean $L^\infty$-error, we see that the $L^\infty$-error reduction factor is included in the above interval of variance reduction ratios. This property may justify a relation between the needed batch-size and the estimators' variance as, e.g., used in \eqref{a5i}.

\subsection{Python code example}
Code of the \textit{level estimator} $P_1-P_0$.

\begin{lstlisting}[language=Python, caption=Training data for level estimator,label={algorithm2}]
def sde_body_p1_p0(idx, s_coarse, s_fine, samples):
    z1 = tf.random_normal(shape=(samples, batch_size_p1_p0, d),
                          stddev=1., dtype=dtype)
    z2 = tf.random_normal(shape=(samples, batch_size_p1_p0, d),
                          stddev=1., dtype=dtype)
    z=(z1+z2)/np.sqrt(2.)
    h_fine=1
    h_coarse=1/2
    s_fine=s_fine + mu *s_fine * h_fine +sigma * s_fine *np.sqrt(h_fine)*z1 + 0.5 *sigma *s_fine *sigma * ((np.sqrt(h_fine)*z1)**2-h_fine)
    s_fine=s_fine + mu *s_fine * h_fine +sigma * s_fine *np.sqrt(h_fine)*z2 + 0.5 *sigma *s_fine *sigma * ((np.sqrt(h_fine)*z2)**2-h_fine)
    s_coarse=s_coarse + mu *s_coarse * h_coarse +sigma * s_coarse *np.sqrt(h_coarse)*z + 0.5 *sigma *s_coarse *sigma * ((np.sqrt(h_coarse)*z)**2-h_coarse)    
    return tf.add(idx, 1), s_coarse, s_fine   
\end{lstlisting}

\section{Proof}\label{proof}

This subsection provides the missing proof.

\begin{proof}[Proof of \cref{corollar2.2}]

We derive an upper bound of $\frac{1}{2} \epsilon^2$ on the square of the expectation of the bias and an upper bound of $\frac{1}{2} \epsilon^2$ on the variance, which together with \eqref{single-mse2} give an $\epsilon^2$ upper bound on \eqref{single-mse}.

Setting $K=32c_3^2\epsilon^{-2}$ and $M=\sqrt{32}c_2\epsilon^{-1}$ together with $\eqref{error-disk-sde} \le \frac{1}{\sqrt{32}}\epsilon, \eqref{error-bestnet} \le \frac{1}{\sqrt{32}}\epsilon$ and $\eqref{error-disk-payoff} =0$ leads to
\begin{align*}
\left(\mathbb{E}\left[\left\lVert \mathbb{E}[P(y)]-\mathcal{N}_{\nu,\theta,\hat{P},\hat{\mathbb{L}}}(y)\right\rVert_{L^1}\right]\right)^2 \le \left( \frac{1}{\sqrt{32}}\epsilon +\frac{1}{\sqrt{32}}\epsilon + 0 + \frac{1}{\sqrt{32}}\epsilon+ \frac{1}{\sqrt{32}}\epsilon \right)^2 = \frac{1}{2}\epsilon^2,
\end{align*}
Furthermore, the above choices for $K$ and $M$ together with the assumptions $\mathbb{V}[\eqref{error-disk-loss}] \le c_2^2 M^{-2}$, and $ \mathbb{V}[\eqref{error-opti}] \le c_3^2 K^{-1}$ lead to
\begin{align*}
\mathbb{V}\left[\left\lVert \mathbb{E}[P(y)]-\mathcal{N}_{\nu,\theta,\hat{P},\hat{\mathbb{L}}}(y)\right\rVert_{L^1}\right]\le \frac{1}{32}\epsilon^2 +\frac{1}{32}\epsilon^2 \le \frac{1}{2} \epsilon^2. 
\end{align*}
The assumption on the complexity is bounded together with the choices for $K$ and $M$ lead to
$C_{\mathcal{N}}\le c_4 32c_3^2 \sqrt{32}c_2\epsilon^{-5}$. Hence, setting $c_5= c_4 32c_3^2 \sqrt{32}c_2$ completes the proof.
\end{proof}

\begin{proof}[Proof of \cref{corollar2.3}]

The proof is analogue to the proof of \cref{corollar2.2}, but we choose $K=32c_3^2\epsilon^{-2}$, $M=32c_2^2\epsilon^{-2}$ and 
\begin{align*}
h=\frac{\epsilon }{c_1 \sqrt{32}},
\end{align*}
which again leads to the $\epsilon^2$ bound on \eqref{single-mse}.
The computational complexity is bounded by
\begin{align*}
C_{\mathcal{N}}\le c_4 K M  h^{-1} \le c_4 32^2c_3^2 c_2^2c_1 \sqrt{32}\epsilon^{-5}.
\end{align*}
Hence, setting $c_5=c_4 32^2c_3^2 c_2^2c_1 \sqrt{32}$ completes the proof.
\end{proof}

To shorten the proof of \cref{complexity-thm} and simplify its presentation, we will use generalized $K_l$ and $M_l$ from the second case onwards, which we apply for all the remaining cases. Nevertheless, this leads to bad estimates for the $\log \epsilon$ terms. Hence, if interested in a particular case, we recommend a more in-depth examination of the first case's proof.

\begin{proof} [Proof of \cref{complexity-thm}]

For each case, we choose specific $K_l$ and $M_l$ such that these bound \eqref{net-mse2}. We use the decomposition \eqref{decomposition} and show that both satisfy an $\frac{1}{2} \epsilon^2$ bound. 
Finally, using the assumptions \eqref{complexity-bound}, we show that the chosen parameters can bound the overall complexity
 \begin{align}
C_{\mathcal{N}} \le \sum_{l=0}^L c_4 h_l^{-1} M_l K_l, \label{complexity-bound2}
\end{align}
with the stated complexity for the specific case.
Let $h_l=H^{-l}T$, for $l=0,\dots,L$ be different step-widths with $H>1$. We start by choosing $L$ to be 
\begin{align*}
L=\left\lceil \frac{\log\left( \sqrt{32}c_1T^\alpha \epsilon^{-1}\right)}{\alpha \log H} \right\rceil,
\end{align*}
so that
\begin{align}
\frac{1}{\sqrt{32}}H^{-\alpha}\epsilon < c_1h_L^\alpha \le \frac{1}{\sqrt{32} }\epsilon.\label{e0ungl}
\end{align}
Let $\theta$ be positive, then
\begin{align*}
\sum_{l=0}^L h_l^\theta&=\sum_{l=0}^L \left(H^{-l}T \right)^\theta
=T^{\theta}\sum_{l=0}^L \left(H^{-\theta} \right)^{l}< \frac{T^{\theta}}{1-H^{-\theta}}.
\end{align*}
On the contrary, for negative exponents we have
\begin{align*}
\sum\limits_{l=0}^L h_l^{-\theta}=h_L^{-\theta} \sum\limits_{l=0}^L (H^\theta)^{-l} & < \frac{H^\theta}{H^\theta-1} h_L^{-\theta},
\end{align*}
which holds due to \eqref{e0ungl} and with 
\begin{align*}
h_L^{-\theta}& < H^\theta \left( \frac{\epsilon}{\sqrt{32}c_1}\right)^{-\theta/\alpha},
\end{align*}
we obtain
\begin{align*}
\sum\limits_{l=0}^L h_l^{-\theta} &< \frac{H^{2 \theta }}{H^\theta-1}\left( \sqrt{32}c_1 \right)^{\theta/\alpha} \epsilon^{-\theta/\alpha}.
\end{align*}
For a simplified presentation, we denote
\begin{align}
g: \theta \mapsto 
\begin{cases}
 \frac{T^{\theta}}{1-H^{-\theta}} &\text{ for } \theta >0, \\
  \frac{H^{2 \theta }}{H^\theta-1}\left( \sqrt{32}c_1 \right)^{\theta/\alpha},& \text{ for } \theta<0. \\
\end{cases} \label{geometric}
\end{align}
I.e., we will have the following inequalities
\begin{align}
\sum_{l=0}^L h_l^\theta \le
\begin{cases}
 g(\theta) &\text{ for } \theta >0, \\
  L+1 &\text{ for } \theta=0,\\
  g(\theta) \epsilon^{-\theta/\alpha},& \text{ for } \theta<0. \\
\end{cases} \label{geometric2}
\end{align}
Now, let us consider the different parameter values.\\

(a) If $\eta=0.5, \gamma=2, \beta=2$ and $\alpha=1$ we set
\begin{align}
K_l = \left\lceil  32 \epsilon^{-2} c_2^2 h_l^{1.5}  g(1.25)^2
\right\rceil \label{kl1}
\end{align}
and
\begin{align}
M_l = \left\lceil  32 \epsilon^{-1} c_3^2c_6^2 h_l^{3/4}  
\right\rceil,\label{ml1}
\end{align}
with 
\begin{align*}
\begin{split}
c_6=\max \left\{g(2/8)(32  c_2^2  g(1.25)^2 )^{-1/4},\right.
\left. T(1-H^{-1/2})^{-2}\left(\sqrt{32  c_2^2   g(1.25)^2}\right)^{-2} \right\},
\end{split}
\end{align*}
Using \eqref{a4i}, \eqref{geometric} and \eqref{kl1}, we obtain
\begin{align*}
\mathbb{E}\left[\sum_{l=0}^L  \left\lVert \mathcal{N}_{\nu,\Theta^l,\hat{Y}_l,{\mathbb{L}}}(y)-\mathcal{N}_{\nu,\theta^l_{K_l},\hat{Y}_l,{\mathbb{L}}}(y))\right\rVert_{L^1}\right]&\le \sum_{l=0}^L  c_2h_l^2 K_l^{-1/2}\\
&\le c_2  \frac{1}{\sqrt{32c_2^2 g(1.25)^2}}\epsilon\sum_{l=0}^L  h_l^{1.25}\notag\\
& \le  \frac{1}{\sqrt{32}} \epsilon .
\end{align*}
Using \eqref{a5i}, \eqref{geometric} and \eqref{ml1}, we obtain 
\begin{align*}
\mathbb{E}\left[\sum_{l=0}^L\left\lVert \mathcal{N}_{\nu,\theta^l_{K_l},\hat{Y}_l,{\mathbb{L}}}(y)-\mathcal{N}_{\nu,\theta^l_{K_l},\hat{Y}_l,{\mathbb{L}}_{M_l}}(y)\right\rVert_{L^1}\right]&\le \sum_{l=0}^L c_3 h_l \rho_l^{1/2} M_l^{-1/2}\\
&\le \sum_{l=0}^L c_3 h_l K_l^{-1/4} \frac{1}{\sqrt{32 \epsilon^{-1} c_3^2c_6^2 h_l^{3/4}}}\\
&\le\frac{1}{\sqrt{32c_6^2}} \epsilon \frac{1}{(32  c_2^2   (g(1.25))^2 )^{1/4}} \sum_{l=0}^L  h_l^{2/8}\\
&\le\frac{1}{\sqrt{32c_6^2}} \epsilon \frac{g(2/8)}{c_6(32  c_2^2  (g(1.25))^2 )^{1/4}} \\
&\le \frac{1}{\sqrt{32}} \epsilon .
\end{align*}
Hence, using these results together with assumptions \eqref{a1}-\eqref{a3} leads to
\begin{align}
\left(\mathbb{E}\left[\left\lVert \mathbb{E}[P(y)]-\hat{\mathcal{N}}(y)\right\rVert_{L^1}\right]\right)^2\le \left( \frac{1}{\sqrt{32} }\epsilon + \frac{1}{\sqrt{32} }\epsilon +\frac{1}{\sqrt{32} }\epsilon + \frac{1}{\sqrt{32} }\epsilon\right)^2 = \frac{1}{2} \epsilon^2.
\end{align}
I.e., we obtain the searched $\frac{1}{2}\epsilon^2$ error bound on the square of the bias. Now, using \eqref{a4ii}, \eqref{kl1} and \eqref{geometric}, we obtain
\begin{align*}
\mathbb{V}\left[\sum_{l=0}^L \left\lVert \mathcal{N}_{\nu,\Theta^l,\hat{Y}_l,{\mathbb{L}}}(y)-\mathcal{N}_{\nu,\theta^l_{K_l},\hat{Y}_l,{\mathbb{L}}}(y))\right\rVert_{L^1}\right]& \le \sum_{l=0}^L  c_2^2 h_l^{4} K_l^{-1}\\
& \le \sum_{l=0}^L  c_2^2 h_l^{4}\frac{1}{32 \epsilon^{-2} c_2^2 h_l^{1.5}  (g(1.25))^2}\\
& = \frac{1}{32} \epsilon^2 \sum_{l=0}^L   h_l^{4}\frac{1}{   h_l^{1.5}  (g(1.25))^2}\\
& \le \frac{1}{32} \epsilon^2   \frac{ g(2.5)}{ g(1.25)^2}\\
& \le \frac{1}{4} \epsilon^2,
\end{align*}
which holds, since we have $g(2.5) < g(1.25)^2$.
Using \eqref{a5ii}, \eqref{ml1} and \eqref{geometric}, we obtain
\begin{align*}
\mathbb{V}\left[\sum_{l=0}^L\left\lVert \mathcal{N}_{\nu,\theta^l_{K_l},\hat{Y}_l,{\mathbb{L}}}(y)-\mathcal{N}_{\nu,\theta^l_{K_l},\hat{Y}_l,{\mathbb{L}}_{M_l}}(y)\right\rVert_{L^1}\right]&\le c_3^2 h_l^{2} \rho_l^{1} M_l^{-1}\notag\\
&\le\epsilon^2 \frac{1}{\sqrt{32  c_2^2   (g(1.25))^2}} \frac{1}{32  c_6^2  }\sum_{l=0}^L  h_l^{1/2}\notag\\
&\le\frac{1}{32  c_6^2  }\epsilon^2 \frac{g(1/2)}{\sqrt{32  c_2^2   (g(1.25))^2}}  \notag\\
& \le  \frac{1}{4} \epsilon^2.
\end{align*}
Hence, we obtain an $\frac{1}{2}\epsilon^2$ upper bound on the variance and the together with the bound on the bias the required $\epsilon^2$ bound on \eqref{net-mse2}. Finally, we study the computational cost for the chosen parameters. Using \eqref{complexity-bound2}, \eqref{geometric}, \eqref{geometric2}, \eqref{kl1}, \eqref{ml1} and the following upper bounds on $K_l$ and $M_l$
\begin{align*}
K_l <  32 \epsilon^{-2} c_2^2 h_l^{1.5}  (g(1.25))^2  +1,
 \hspace{10mm} 32 \epsilon^{-1} c_3^2c_6^2 h_l^{3/4}    +1,
\end{align*}
we obtain 
\begin{align*}
C_{\mathcal{N}} &\le \sum_{l=0}^L c_4 h_l^{-1} M_l K_l\\
&\le \sum_{l=0}^L c_4 h_l^{-1}\left(32 \epsilon^{-2} c_2^2 h_l^{1.5}  (g(1.25))^2  +1\right) \left( 32 \epsilon^{-1} c_3^2c_6^2 h_l^{3/4}    +1\right)\\
&= c_4 \sum_{l=0}^L  h_l^{-1}\left( c_7\epsilon^{-2}  \left( h_l\right)^{1.5}+1\right) \left( c_8 \epsilon^{-1}  h_l^{0.75}+1\right)\\
&\le c_4c_7c_8\epsilon^{-3} g(1.25) + c_4c_7\epsilon^{-2} g(0.5) + c_4c_8\epsilon^{-1} g(-0.25) \epsilon^{-0.25/\alpha} + c_4 g(-1) \epsilon^{-1/\alpha},
\end{align*}
with $c_7=32 c_2^2   (g(1.25))^2$ and $c_8=32 c_3^2c_6^2$.
Hence, for this case we obtain the required complexity bound
\begin{align*}
C_{\mathcal{N}}&\le c_5 \epsilon^{-3}, 
\end{align*}
with $c_5=c_4c_7c_8 g(1.25)+c_4c_7 g(0.5)+c_4 c_8 g(-0.25) + c_4 g(-1) $, which completes the proof.
Now, as mentioned above, instead of individual choices for $K_l$ and $M_l$ for the remaining cases, we will choose them in a more generic way, i.e. let
\begin{align}
K_l = \left\lceil  32 \epsilon^{-2} c_2^2 h_l^{2 \gamma\alpha} (L+1)^2 
\right\rceil \label{Kl}
\end{align}
and 
\begin{align}
M_l = \left\lceil  32 \epsilon^{-2 +2\eta} c_3^2 h_l^{\bar{\beta}}  (L+1)^2
\right\rceil. \label{Ml}
\end{align}
We will use \eqref{Kl} and \eqref{Ml} for each of the remaining cases. As we will see in the following, these choices satisfy the required error bound on the MSE for each case. Hence, for each case, we will only have to study the individual computational cost.
Using \eqref{a4i} and \eqref{Kl}, we obtain
\begin{align}
\mathbb{E}\left[\sum_{l=0}^L \left\lVert \mathcal{N}_{\nu,\Theta^l,\hat{Y}_l,{\mathbb{L}}}(y)-\mathcal{N}_{\nu,\theta^l_{K_l},\hat{Y}_l,{\mathbb{L}}}(y))\right\rVert_{L^1}\right]&\le \sum_{l=0}^L c_2 h_l^{\alpha\gamma} K_l^{-1/2}\notag\\
& \le \frac{1}{\sqrt{32}} \epsilon\label{e4ungl} .
\end{align}
Using \eqref{a5i} and \eqref{Ml}, we obtain
\begin{align}
\mathbb{E}\left[\sum_{l=0}^L\left\lVert \mathcal{N}_{\nu,\theta^l_{K_l},\hat{Y}_l,{\mathbb{L}}}(y)-\mathcal{N}_{\nu,\theta^l_{K_l},\hat{Y}_l,{\mathbb{L}}_{M_l}}(y)\right\rVert_{L^1}\right]&\le \sum_{l=0}^L c_3 h_l^{\beta/2} \rho_l^{\eta} M_l^{-1/2}\notag\\
&\le \sum_{l=0}^L c_3 h_l^{\beta/2} K^{-\eta/2} M_l^{-1/2}\notag\\
&\le  \frac{1}{\sqrt{32}(L+1)}\epsilon^{\eta+1-\eta}\sum_{l=0}^L h_l^{\beta/2} \frac{1}{  h_l^{ \gamma\alpha \eta}     h_l^{\bar{\beta/2}}  }\notag\\
&\le \frac{1}{\sqrt{32}}\epsilon\label{e5ungl},
\end{align}
where the last inequality holds by choosing $\bar{\beta}=\beta-2\alpha\gamma\eta$.
For the variance analogue formulas hold, as we will see in the following calculations.
Using \eqref{a4ii} and \eqref{Kl}, we obtain
\begin{align}
\mathbb{V}\left[ \sum_{l=0}^L\left\lVert \mathcal{N}_{\nu,\Theta^l,\hat{Y}_l,{\mathbb{L}}}(y)-\mathcal{N}_{\nu,\theta^l_{K_l},\hat{Y}_l,{\mathbb{L}}}(y))\right\rVert_{L^1}\right]& \le \sum_{l=0}^L  c_2^2 h_l^{2\alpha\gamma} K_l^{-1}\notag\\
&\le \sum_{l=0}^L    \frac{1}{ 32 \epsilon^{-2}  (L+1)^2 }\notag\\
& \le \frac{1}{4} \epsilon^2\label{e4varungl}.
\end{align}
Using \eqref{a5ii} and \eqref{Ml}, we obtain
\begin{align}
\mathbb{V}\left[ \sum_{l=0}^L \left\lVert \mathcal{N}_{\nu,\theta^l_{K_l},\hat{Y}_l,{\mathbb{L}}}(y)-\mathcal{N}_{\nu,\theta^l_{K_l},\hat{Y}_l,{\mathbb{L}}_{M_l}}(y)\right\rVert_{L^1}\right]&\le \sum_{l=0}^Lc_3^2 h_l^{\beta} \rho_l^{2\eta} M_l^{-1}\notag\\
&\le \frac{1}{32(L+1)^2}\epsilon^{2\eta+2-2\eta}\sum_{l=0}^L h_l^{\beta} \frac{1}{  h_l^{2 \gamma\alpha\eta} {}   h_l^{\bar{\beta}}  }\notag\\
& \le \frac{1}{4} \epsilon^2\label{e5varungl},
\end{align}
again the last inequality holds by choosing $\bar{\beta}=\beta-2\alpha\gamma\eta$.

Hence, we obtain an $\frac{1}{2}\epsilon^2$ upper bound on the variance and with \eqref{e4ungl} and \eqref{e5ungl} an $\frac{1}{2}\epsilon^2$ upper bound on the bias. Together, the required we obain a $\epsilon^2$ bound on the MSE \eqref{net-mse2}.

Now, we will study the computational cost for each case.
We have the following upper bounds
\begin{align*}
K_l<32  \epsilon^{-2} c_2^2 h_l^{2 \gamma\alpha} (L+1)^2   +1
\end{align*}
and
\begin{align*}
M_l< 32 \epsilon^{-2 +2\eta} c_3^2 h_l^{\bar{\beta}}  (L+1)^2 +1.
\end{align*}
Together with \eqref{complexity-bound2} this leads to 
\begin{align}
C_{\mathcal{N}} &\le \sum_{l=0}^L c_4 h_l^{-1} \left( 32 \epsilon^{-2} c_2^2 h_l^{2 \gamma\alpha} (L+1)^2   +1\right) \left(32 \epsilon^{-2 +2\eta} c_3^2 h_l^{\bar{\beta}}  (L+1)^2 +1\right)\notag\\
\begin{split}&\le c_4c_7c_8 (L+1)^4  \epsilon^{-4 +2\eta}\sum_{l=0}^L  h_l^{-1+2 \gamma\alpha+\bar{\beta}}   
  + c_4c_7(L+1)^2\epsilon^{-2}\sum_{l=0}^L h_l^{-1+2 \gamma\alpha}  \\ 
  &   + c_4c_8(L+1)^2\epsilon^{-2 +2\eta}\sum_{l=0}^L h_l^{-1+\bar{\beta}}       + c_4 g(-1)\epsilon^{-1/\alpha}, \end{split}\label{complex} 
\end{align}
with the constants $c_7=32 c_2^2$ and $c_8=32 c_3^2$.
For $L$ of \eqref{complex} we have 
\begin{align*}
L \le \frac{\log \epsilon^{-1}}{\alpha \log H} + \frac{\log (\sqrt{2}c_1 T^\alpha}{\alpha \log H}+1
\end{align*}
and since $1< \log\epsilon^{-1}$ for $\epsilon< \exp(-1)$ it follows that 
\begin{align}
L+1 \le c_9\log \epsilon^{-1},\label{hilf4}
\end{align}
where
\begin{align*}
c_9= \frac{1}{\alpha \log H}+ \max \left( 0,\frac{\log (\sqrt{2}c_1 T^\alpha}{\alpha \log H} \right) +2.
\end{align*}
Let us consider the different parameter values. Again, we will use $\bar{\beta}=\beta-2\alpha\gamma\eta$ for each case.\\

(b) Let $\eta=0, \gamma=0, \beta=0$ and $\alpha=1/2$. Then, $\bar{\beta}=0$ and with \eqref{complex} we have
\begin{align*}
C_{\mathcal{N}}&
\le c_{6} \epsilon^{-4}(L+1)^4\sum_{l=0}^L h_l^{-1},
\end{align*}
with $c_6=c_4(c_7c_8+c_7+c_8+1)$.
Using \eqref{geometric}, \eqref{geometric2} and \eqref{hilf4} we obtain
\begin{align*}
C_{\mathcal{N}}\le c_{6} \epsilon^{-6} | \log \epsilon |^4,
\end{align*}
with $c_5=c_6c_9^4g(-1)$.
(c) Let $\eta=0, \gamma=0, \beta=0$ and $\alpha=1$. Then, $\bar{\beta}=0$ and with \eqref{complex} we have
\begin{align*}
C_{\mathcal{N}}&
\le c_{5} \epsilon^{-4}(L+1)^4\sum_{l=0}^L h_l^{-1},
\end{align*}
with $c_6=c_4(c_7c_8+c_7+c_8+1)$.
Using \eqref{geometric}, \eqref{geometric2} and \eqref{hilf4} we obtain
\begin{align*}
C_{\mathcal{N}}\le c_{6} \epsilon^{-5} | \log \epsilon |^4,
\end{align*}
with $c_5=c_6c_9^4g(-1)$.

(d) Let $\eta=0, \gamma=0, \beta=1$ and $\alpha=1$. Then, $\bar{\beta}=1$ and with \eqref{complex} we have
\begin{align*}
C_{\mathcal{N}}&
\le c_{6} \epsilon^{-4}(L+1)^4 \sum_{l=0}^L h_l^{0},
\end{align*}
with $c_6=c_4(c_7c_8+c_7+c_8+1)$.
Using \eqref{geometric}, \eqref{geometric2} and \eqref{hilf4} we obtain
\begin{align*}
C_{\mathcal{N}}\le c_{5} \epsilon^{-4} | \log \epsilon |^5,
\end{align*}
with $c_5=c_6c_9^5$.

(e) Let $\eta=0.25, \gamma=1, \beta=1$ and $\alpha=1$. Then, $\bar{\beta}=0.5$ and with \eqref{complex} we have
\begin{align*}
C_{\mathcal{N}}&
\le c_{6} \epsilon^{-3.5}(L+1)^4 \sum_{l=0}^L h_l^{1.5},
\end{align*}
with $c_6=c_4(c_7c_8+c_7+c_8+1)$.
Using \eqref{geometric}, \eqref{geometric2} and \eqref{hilf4} we obtain
\begin{align*}
C_{\mathcal{N}}\le c_{5} \epsilon^{-3.5} | \log \epsilon |^4,
\end{align*}
with $c_5=c_6c_9^4g(1.5)$.
(f) Let $\eta=0.5, \gamma=1, \beta=1$ and $\alpha=1$. Then, $\bar{\beta}=0$ and with \eqref{complex} we have
\begin{align*}
C_{\mathcal{N}}&
\le c_{6} \epsilon^{-3}(L+1)^4 \sum_{l=0}^L h_l^{1},
\end{align*}
with $c_6=c_4(c_7c_8+c_7+c_8+1)$.
Using \eqref{geometric}, \eqref{geometric2} and \eqref{hilf4} we obtain
\begin{align*}
C_{\mathcal{N}}\le c_{5} \epsilon^{-3} | \log \epsilon |^4,
\end{align*}
with $c_5=c_6c_9^4g(1)$.

\end{proof}

\vskip 0.2in

\end{document}